\documentclass{amsart}
\usepackage{amssymb,mathrsfs,mathptmx,txfonts,enumerate,tensor,amsaddr,mathtools}
\mathtoolsset{showonlyrefs,showmanualtags}

\newtheorem{thm}{Theorem}[section] \newtheorem{lemma}[thm]{Lemma}
\newtheorem{cor}[thm]{Corollary} \newtheorem{prop}[thm]{Proposition}
\theoremstyle{definition}

\theoremstyle{remark}

\newcommand\newdot{{\kern.8pt\cdot\kern.8pt}}
\newcommand\nbull{{\kern.8pt\raise1.5pt\hbox{\bf .}\kern.8pt}}

\font\sevenrm=cmr7

\newcommand\E{\mathbb{E}}

\newcommand\R{\mathbb{R}}

\newbox\ovlbox
\def\ovl#1{\setbox\ovlbox\hbox{$#1$}\rlap{\kern.5\wd\ovlbox\kern-1.5pt
  $\overline{\hbox to4pt{\hss$\phantom{#1}$\hss}}$\hss}#1}
\def\unl#1{\setbox\ovlbox\hbox{$#1$}\rlap{\kern.5\wd\ovlbox\kern-2.5pt
  $\underline{\hbox to4pt{\hss$\phantom{#1}$\hss}}$\hss}#1}

\catcode`\@=11
\def\acong{\mathrel{\mathpalette\@avereq\sim}} 
\def\@avereq#1#2{\lower.5\p@\vbox{\baselineskip\z@skip\lineskip-.5\p@
    \ialign{$\m@th#1\hfil##\hfil$\crcr#2\crcr\longrightarrow\crcr}}}
\def\mequal{\mathrel{\mathpalette\@mvereq{\hbox{\sevenrm m}}}}
\def\@mvereq#1#2{\lower.5\p@\vbox{\baselineskip\z@skip\lineskip1.5\p@
    \ialign{$\m@th#1\hfil##\hfil$\crcr#2\crcr=\crcr}}}
\def\relop#1#2#3{\mathrel{\mathop{\kern\z@ #2}\limits^{#1}_{#3}}}

\newcommand\1{\hbox{\kern.375em\vrule height1.57ex depth-.1ex
    width.05em\kern-.375em \rm 1}}

\def\vol{{\operatorname{vol}}} 
\def\Hess{{\operatorname{Hess}}}\def\Ric{{\operatorname{Ric}}}
\def\End{{\operatorname{End}}}\def\inj{{\operatorname{inj}}}
\def\e{\operatorname{e}}

 \DeclareMathOperator{\trace}{trace}

\def\mathpal#1{\mathop{\mathchoice{\text{\rm #1}}%
    {\text{\rm #1}}{\text{\rm #1}}%
    {\text{\rm #1}}}\nolimits} \def\id{{\mathpal{id}}}
\def\boxit#1{\vbox{\hrule\hbox{\vrule\kern3pt
      \vbox{\kern3pt#1\kern3pt}\kern3pt\vrule}\hrule}}

\let\oldtocsection=\tocsection \let\oldtocsubsection=\tocsubsection
\renewcommand{\tocsection}[2]{\hspace{0em}\oldtocsection{#1}{#2}}
\renewcommand{\tocsubsection}[2]{\hspace{2em}\oldtocsubsection{#1}{#2}}

\numberwithin{equation}{section}

\begin{document}
	\title[Quantitative $C^1$-estimates]{Quantitative $C^1$-estimates by Bismut Formulae}
  
   \author{Li-Juan Cheng\textsuperscript{1,2}, Anton
    Thalmaier\textsuperscript{1} and James Thompson\textsuperscript{1}}

  \address{\textsuperscript{1}Mathematics Research Unit, FSTC, University of Luxembourg\\
    Maison du Nombre, 6, Avenue de la Fonte,\\ 4364 Esch-sur-Alzette, Grand Duchy of Luxembourg}
  \address{\textsuperscript{2}Department of Applied Mathematics, Zhejiang University of Technology\\
    Hangzhou 310023, The People's Republic of China}
  \email{lijuan.cheng@uni.lu \text{and} chenglj@zjut.edu.cn}
  \email{anton.thalmaier@uni.lu}
  \email{james.thompson@uni.lu}

\begin{abstract}\noindent
For a $C^2$ function $u$ and an elliptic operator $L$, we prove a quantitative estimate for the derivative $du$ in terms of local bounds on $u$ and $Lu$. An integral version of this estimate is then used to derive a condition for the zero-mean value property of $\Delta u$. An extension to differential forms is also given. Our approach is probabilistic and could easily be adapted to other settings.
\end{abstract}

 \date\today

\maketitle

\section{Introduction}

\subsection{}
Suppose $M$ is a complete and connected Riemannian manifold of dimension $n$. Denote by $\rho$ the Riemannian distance function. Denote by $\nabla$ the Levi-Civita connection, by $\Delta$ the Laplace-Beltrami operator and for a smooth vector field $Z$ consider the elliptic operator $L := \frac12\Delta +Z$. We will prove that if $u \in C^2(M)$ then for any regular domains $D_0 \subset D$ with $r_0 := \rho(D_0,\partial D)$ and $\delta >0$ we have
\begin{equation}\label{eq:mainres}
\sup_{D_0}|du| \leq C \left( \frac{2\delta}{r_0} \sup_D |u|  + \frac{r_0}{2\delta} \sup_D |2Lu| \right)
\end{equation}
for an explicit constant $C$. While it is straightforward to obtain such an estimate without any control on the constant, our constant is explicit and depends on the geometry of $D$ only via a lower bound on Ricci curvature (and some assumptions on $Z$; for the precise form of the constant, see Theorem \ref{thm:mainthm}). An estimate of this type was recently proved using analytic methods by G\"{u}neysu and Pigola in \cite{GP2017}, but with the constant depending on sectional curvature.

Our approach, on the other hand, is probabilistic in nature. Probabilistic approaches to gradient estimates are typically based either on a derivative formula, such as the original one proved by Bismut \cite{bismut}, or coupling methods, as introduced by Cranston \cite{Cranston}. One advantage of these stochastic methods is the ease with which they can be used to obtain local estimates with explicit constants, such as for example the gradient estimates for harmonic functions of \cite{TW,ADT}, the Li-Yau type estimates and Harnack inequalities of \cite{AT,ATW} and a range of other functional inequalities for solutions to the heat equation with various boundary conditions; see the monograph \cite{Wangbook} and references therein for further reading. Although there is indeed a long history of using stochastic analysis to study such things, it is nonetheless surprising that explicit estimates of the type \eqref{eq:mainres} can be obtained from the stochastic analysis of Brownian motion in such a simple and straight-forward way.

Our new approach yields eigenfunction estimates (see Corollary \ref{cor:mainthm2}) and can be extended to differential forms (see Theorem \ref{thm:estforforms} and Corollary \ref{cor:eigestforforms} in Section \ref{sec:forms}) and vector bundles. The same approach can also be used to obtain global integral estimates for the symmetric case, which we present in Section \ref{sec:intestimates}. In particular, we will prove for each $p>1$ that if the Ricci curvature is bounded below then
\begin{equation}\label{eq:mainres2}
\| du \|_{L^p(M)} \leq C(p) \left( 2\delta\, \| u \|_{L^p(M)} + \frac{1}{2\delta} \, \| \Delta u \|_{L^p(M)}\right)
\end{equation}
where the constant $C(p)$ is again explicit (see Theorem \ref{thm:mainthm3}), although as explained in \cite{GPlp} the Ricci curvature lower bound is actually redundant if $p \in (1,2]$.

G\"{u}neysu and Pigola observed in \cite{GP2017} that an estimate of the type \eqref{eq:mainres} can be used to derive conditions for the vanishing of the integral of $\Delta u$. This approach involves sublinear volume growth and a control on sectional curvature. An alternative set of conditions, avoiding these restrictions, can be obtained using the integral estimate \eqref{eq:mainres2}. In particular, we obtain from Theorem \ref{thm:mainthm3} and Karp's divergence theorem \cite{Karp} the following corollary, for which we denote by $B_r$ a ball of radius $r$:

\begin{cor}\label{cor:meanzero}
Suppose there exist $c>0$, $q>1$ with $\vol(B_r) \leq cr^q$ for $r \geq 1$ and that the Ricci curvature of $M$ is bounded below. Suppose $u \in C^2(M)$ and that $\Delta u$ has an integral (i.e. either $(\Delta u)^+$ or $(\Delta u)^-$ is integrable) with $u,\,\Delta u \in L^p(M)$ for $1/p+1/q=1$. Then $$\int_M \Delta u(x)\,dx = 0.$$
\end{cor}

In particular, under the assumptions of Corollary \ref{cor:meanzero} it follows that if $f\in C^2(M)$ is a harmonic function with $f \in L^{2p}(M)$ for $1/p+1/q=1$ then $f$ is constant (apply Corollary \ref{cor:meanzero} to $f^2$ and use \eqref{eq:mainres2} to verify $\Delta f^2 = 2|d f|^2 \in L^p(M)$).

\subsection{}
Before proving the main results, let us start with a simple, suggestive calculation. Fix $r>0$, suppose $u \in C^2(M)$, $x \in M$, $v \in T_xM$ with $|v|=1$, set $\gamma_v(s) = \exp (sv)$ for $s \in [0,r]$ and consider $w(s):=u(\gamma_v(s))$. Then, by Taylor's theorem, we have
\begin{equation}
w(s) = w(0) +sw'(0) + \int_0^s (s-t)w''(t)dt.
\end{equation}
Calculating $w'(0)$ and $w''(t)$ for the geodesic $\gamma_v$ yields
\begin{equation}
s(du)_x(v) = u(\gamma_v(s)) - u(x) - \int_0^s (s-t)(\Hess u)_{\gamma_v(t)}(\dot{\gamma}_v(t),\dot{\gamma}_v(t))dt
\end{equation}
which implies
\begin{equation}
|(du)_x(v)| \leq \frac2{s} \sup_{\gamma_v[0,s]} |u| + \frac{s}2\sup_{\gamma_v[0,s]} |\Hess u|
\end{equation}
for all $s \in (0,r]$. Consequently, if $D_0 \subset D$ are regular domains (open and relatively compact with non-empty smooth boundary) with $r_0 = \rho(D_0,\partial D)$, then
\begin{equation}
\sup_{D_0}|du| \leq \frac2{r_0} \sup_{D} |u| + \frac{r_0}2\sup_{D} |\Hess u|.
\end{equation}
Replacing the Hessian with the Laplacian requires, according to Bochner's formula, a lower bound on the Ricci curvature of $D$. To obtain precisely the estimate \eqref{eq:mainres}, however, the strategy will be to use Ito's formula instead of Taylor's theorem and replace the geodesic path with a Brownian motion or, more generally, an $L$-diffusion.

\section{Main results}\label{sec:functions}

\subsection{}
Suppose $D$ is a regular domain with $X(x)$ an $L$-diffusion starting at $x \in D$ with stochastic parallel transport $//$ and whose anti-development to $T_xM$ has martingale part $B$. Denote by $\tau$ the first exit time of $X(x)$ from $D$. Then, by It\^{o}'s formula, we have for $u \in C^2(M)$ that
\begin{equation}
u(X_{t\wedge \tau}(x)) = u(x) + \int_0^{t\wedge \tau} (du)_{X_s(x)}(//_s dB_s) + \int_0^{t \wedge \tau} Lu(X_s(x))\,ds
\end{equation}
and therefore
\begin{equation}\label{eq:expito}
\E\left[ u(X_{t\wedge \tau}(x))\right] = u(x) + \int_0^t \E \left[ 1_{\lbrace s < \tau \rbrace} (Lu)(X_s(x))\right] ds.
\end{equation}
Note that this equation involves two slightly different semigroups. First there is
\begin{equation}
P^1_tu(x) := \E\left[ u(X_{t\wedge \tau}(x))\right]
\end{equation}
which solves the diffusion equation $(\partial_t - L)u_t = 0$ on $D$ with initial condition $u_0 = u$ and boundary condition $u_t \vert_{\partial D} = u\vert_{\partial D}$. Second there is
\begin{equation}
P^2_tu(x) := \E\left[ 1_{\lbrace t < \tau \rbrace} u(X_t(x))\right]
\end{equation}
which also solves the diffusion equation on $D$ with initial condition $u_0 = u$, but with the Dirichlet boundary condition $u_t \vert_{\partial D} = 0$. These facts are easily verified by It\^{o}'s formula. Equation \eqref{eq:expito} can therefore be rearranged as
\begin{equation}\label{eq:expito2}
u(x) = P^1_tu(x) - \int_0^t P^2_s(Lu)(x) \,ds.
\end{equation}
Our main result, inequality \eqref{eq:mainres}, is obtained from this by differentiating both sides and applying the Bismut formula, for the semigroups $P^1$ and $P^2$. The Bismut formula was introduced on compact manifolds by Bismut in \cite{bismut} and extended by Elworthy and Li in \cite{ElworthyLi}. The version that we use allows localization and was introduced by the second author in \cite{Thalmaier97}. The details of all this will now be explained.

\subsection{}
Denote by $\Ric^Z = \Ric -2\nabla Z$ the Bakry-Emery tensor and by $\mathscr{Q}$ the $\End(T_xM)$-valued solution to the ordinary differential equation
\begin{equation}
\frac{d}{dt} \mathscr{Q}_t  = -\frac12\Ric^Z_{//_t}\mathscr{Q}_t
\end{equation}
with $\mathscr{Q}_0 = \id_{T_xM}$ and $\Ric^Z_{//_t} := //_t^{-1} \Ric^Z //_t$. Fix $t>0$ and suppose $h$ is a bounded adapted process with paths in the Cameron-Martin space $L^{1,2}([0,t];\R)$ such that $h(s) = 0$ for $s \geq t \wedge \tau$. If $u_s$ is a solution to the diffusion equation on $D$ then, by It\^{o}'s formula and the Weitzenb\"{o}ck formula, it follows that
\begin{equation}
du_{t-s}(//_s \mathscr{Q}_s h(s) ) - u_{t-s}(X_s(x))\int_0^s \langle \mathscr{Q}_r \dot{h}(r),dB_r\rangle
\end{equation}
is a local martingale. If in addition $h(0) = 1$ with $\int_0^{t \wedge \tau} |\dot{h}(s)|^2 ds$ integrable, then evaluating at times $0$ and $s = t\wedge \tau$, taking expectations and using the initial and boundary conditions we obtain
\begin{align}
(dP^1_tu)_x &= -\E\left[ u(X_{t\wedge \tau}(x)) \int_0^{t\wedge \tau} \langle \mathscr{Q}_r \dot{h}(r),dB_r\rangle\right],\label{eq:bisone} \\ 
(dP^2_t u)_x &= -\E\left[ 1_{\lbrace t < \tau \rbrace} u(X_t(x)) \int_0^{t} \langle \mathscr{Q}_r \dot{h}(r),dB_r\rangle\right]. \label{eq:bistwo}
\end{align}
These Bismut formulae express the derivatives of the semigroups in a way which does not involve the derivatives of $u$. Following \cite{ThalmaierWang2011}, an explicit choice of $h$ allows to obtain explicit estimates:

\begin{lemma}\label{lem:est}
Suppose $\phi \in C^2(\ovl{D})$ with $\phi >0$ and $\phi \leq 1$ on $D$, $\phi(x)=1$, $\phi\vert_{\partial D} = 0$ and set
\begin{equation}
c(\phi) := \sup_D \Big\lbrace 3|\nabla \phi|^2 -2 \phi L \phi \Big\rbrace.
\end{equation}
Then there exists a bounded adapted process $h$ with paths in the Cameron-Martin space $L^{1,2}([0,t];\R)$ such that $h(0) = 1$, $h(s) = 0$ for $s \geq t \wedge \tau$ and
\begin{equation}
\E \left[\int_0^{t \wedge \tau} \dot{h}^2(s) ds\right] \leq  \frac{c(\phi)}{1-\e^{-c(\phi)t}}.
\end{equation}
\end{lemma}

\begin{proof}
Consider the time change
\begin{equation}
\sigma(s) = \inf \Bigg\lbrace r \geq 0 : \int_0^r \phi^{-2}(X_u(x))du \geq s \Bigg\rbrace
\end{equation}
and let
\begin{equation}
h_0(s) = \int_0^s \phi^{-2}(X_r(x))1_{\lbrace r < \sigma(t) \rbrace}\,dr.
\end{equation}
Finally let $h(s) = (h_1 \circ h_0)(s)$ where $h_1 \in C^1([0,t])$ is chosen so that $h_1(0) = 1$, $h_1(t) = 0$ and $\dot{h}_1 \leq 0$. Then, as in \cite[Remark~3.2]{ThalmaierWang2011}, it follows that $h(0)=1$, $h(s) = 0$ for $s\geq \sigma(t)$ with $\sigma(t) \leq \tau \wedge t$ and
\begin{equation}
\E \left[ \int_0^{t \wedge \tau} \dot{h}^2(s) ds \right] = \E \left[ \int_0^{\sigma(t)} (\dot{h}_1  \circ h_0)^2(s) \phi^{-4}(X_s(x))\,ds \right] = \int_0^t \dot{h}^2_1(s) \E\left[ \phi^{-2}(X'_s(x))\right]ds
\end{equation}
where $X'(x)$ denotes the diffusion starting at $x$ with generator $\phi^2L$ which almost surely does not exit $D$ by \cite[Prop.~2.3]{TW}. To estimate the integrand we use
\begin{equation}
\phi^2 L \phi^{-2} = \left(3|\nabla \phi|^2 -2 \phi L \phi\right)\phi^{-2}
\end{equation}
to obtain, via It\^{o}'s formula and Gronwall's lemma, that
\begin{equation}
\E\left[ \phi^{-2}(X'_s(x))\right] \leq \phi^{-2}(x) \,\e^{c(\phi) s}.
\end{equation}
Using $\phi(x)=1$ and taking
\begin{equation}
h_1(s) = 1-\frac{c(\phi)}{1-\e^{-c(\phi)t}}\int_0^s \e^{-c(\phi)r} dr
\end{equation}
we obtain the desired estimate.
\end{proof}

Now set $K_0:= \inf \left\lbrace \Ric(v,v):v\in TD,\ |v|=1\right\rbrace$.

\begin{lemma}\label{lem:lemball}
Suppose $D$ is a ball of radius $r$ centred at $x$. Then there exists $\phi$ satisfying the conditions of Lemma \ref{lem:est} with
\begin{equation}
c(\phi) \leq \frac{\pi}{2 r} \left( 2 \sup_D |Z|+\sqrt{(n-1)K^-_0}\right) + \frac{\pi^2 (n+3)}{4r^2}.
\end{equation}
\end{lemma}

\begin{proof}
Take
\begin{equation}
\phi(p) = \cos \left( \frac{\pi \rho(x,p)}{2 r}\right).
\end{equation}
Clearly this choice of $\phi$ satisfies the conditions of Lemma \ref{lem:est}. Furthermore
\begin{equation}
|\nabla \phi| \leq \frac{\pi}{2 r}
\end{equation}
and by the Laplacian comparison theorem
\begin{equation}
-\Delta \phi \leq \frac{\pi}{2 r}\sqrt{(n-1)K^-_0} + \frac{\pi^2 n}{4r^2}
\end{equation}
which together give the estimate on $c(\phi)$.
\end{proof}

Now set $K_Z:=\inf \lbrace \Ric^Z(v,v):v\in TD,\ |v|=1\rbrace$.

\begin{prop}\label{prop:prelimest}
Suppose $D_0 \subset D$ are regular domains with $u \in C^2(M)$. Set $r_0 = \rho(D_0,\partial D)$. Then
\begin{equation}
|dP^1_tu|(x) \vee |dP^2_tu|(x) \leq \sup_D |u| \frac1{\sqrt{t}} \left(\frac{ct\e^{K_Z^- t}}{1-\e^{-ct}}\right)^{1/2}
\end{equation}
for all $t> 0$ and $x \in D_0$ where
\begin{equation}
c := \frac{\pi}{2r_0} \left( 2 \sup_D |Z|+\sqrt{(n-1)K^-_0}\right) + \frac{\pi^2 (n+3)}{4r^2_0}.
\end{equation}
\end{prop}

\begin{proof}
Formulae \eqref{eq:bisone} and \eqref{eq:bistwo} hold for the ball of radius $r_0$ centred at $x$, so the result follows by Lemmas \ref{lem:est} and \ref{lem:lemball}.
\end{proof}

\begin{thm}\label{thm:mainthm}
Suppose $D_0 \subset D$ are regular domains with $u \in C^2(M)$. Set $r_0 = \rho(D_0,\partial D)$ and $\delta >0$. Then
\begin{equation}\label{eq:mainthmest}
\sup_{D_0}|du| \leq C \left( \frac{2\delta}{r_0} \sup_D |u|  + \frac{r_0}{2\delta} \sup_D |2Lu| \right)
\end{equation}
where
\begin{equation}
C:= \exp\left[\frac{\pi r_0}{16\delta^2} \left( 2 \sup_D |Z|+\sqrt{(n-1)K^-_0}\right) + \frac{\pi^2 (n+3)}{32\delta^2}+\frac{r_0^2 K_Z^-}{8\delta^2}\right].
\end{equation}
\end{thm}

\begin{proof}
Differentiating both sides of equation \eqref{eq:expito2} gives
\begin{equation}
du_x  = dP^1_tu_x - \int_0^t dP^2_s(Lu)_x \,ds
\end{equation}
and therefore
\begin{equation}
|du|(x) \leq \left|dP^1_tu\right|(x) + \int_0^t \left|dP^2_s(Lu)\right|(x) \,ds.
\end{equation}
Applying Proposition \ref{prop:prelimest}, it follows that
\begin{equation}
\begin{split}
|du|(x) &\leq \sup_D |u| \frac1{\sqrt{t}} \left(\frac{ct\e^{K_Z^- t}}{1-\e^{-ct}}\right)^{1/2} + \sup_D |Lu|\int_0^t \frac1{\sqrt{s}} \left(\frac{cs\e^{K_Z^- s}}{1-\e^{-cs}}\right)^{1/2} ds \\
&\leq \e^{(c+K_Z^-) \,t/2}\left( \frac1{\sqrt{t}} \sup_D |u|  + \sqrt{t} \sup_D |2Lu| \right).
\end{split}
\end{equation}
The result follows by setting $t = r_0^2/4\delta^2$.
\end{proof}

Note that setting $\delta^2 = (1\vee r_0)^2(1\vee K^-)$ in Theorem \ref{thm:mainthm} with $Z=0$ implies
\begin{equation}
\sup_{D_0}|du| \leq C(n) \sqrt{1\vee K^-} \left( \frac{1\vee r_0}{r_0}\right) \left( \sup_D |u|  + \sup_D |\Delta u| \right).
\end{equation}
We therefore recover the behaviour, in $r_0$ and the curvature bound, of the constant that was previously obtained by G\"{u}neysu and Pigola in \cite{GP2017}. We also therefore recover a stronger version of the zero mean value condition \cite[Corollary~1.3]{GP2017} in which the supremum norm of sectional curvature can be replaced by only a lower bound on Ricci curvature. Note that \cite[Corollary~1.3]{GP2017} requires sublinear volume growth at infinity and can not, therefore, be applied even to the case $M=\mathbb{R}^n$. In Section \ref{sec:intestimates} we will investigate circumstances under which this volume growth condition can be relaxed.

\subsection{}
As a further corollary to Theorem \ref{thm:mainthm}, we obtain quantitative estimates on the gradient of eigenfunctions:

\begin{cor}\label{cor:mainthm2}
Suppose $D_0 \subset D$ are regular domains with $u \in C^2(M)$ satisfying $2Lu = -\lambda u$ for some $\lambda >0$. Then
\begin{equation}
\sup_{D_0} |du| \leq \sqrt{\lambda}\, C^{1/\lambda} \left(\frac{2\delta}{r_0}+\frac{r_0}{2\delta}\right)\sup_D |u|
\end{equation}
where the constant $C$ is given as in Theorem \ref{thm:mainthm}.
\end{cor}

\begin{proof}
Replace $\delta$ with $\delta\sqrt{\lambda}$ in Theorem \ref{thm:mainthm} and use $2Lu=-\lambda u$.
\end{proof}

If the Ricci curvature of $M$ is bounded below then Brownian motion is non-explosive. There is the following probabilistic formula for the gradient of an eigenfunction of the Laplacian, which is a consequence of Bismut's formula, but the proof of which we include for completeness:

\begin{prop}\label{prop:nonlocformone}
Suppose $u \in C^2(M)$ with $u$ bounded and satisfying $\Delta u = -\lambda u$ for some $\lambda >0$. Suppose $\Ric$ is bounded below. For $x\in M$ suppose $X(x)$ is a Brownian motion on $M$ starting at $x$. Then for $v \in T_{x}M$ we have
\begin{equation}
(du)(v)  = \frac{\lambda \e}2\mathbb{E}\left[ u(X_{2/\lambda}(x))\int_0^{2/\lambda} \langle \mathscr{Q}_s v,dB_s\rangle\right].
\end{equation}
\end{prop}

\begin{proof}
Since $\Ric$ is bounded below, it follows from Corollary \ref{cor:mainthm2} that there exists a positive constant $C(\lambda)$ such that $|du|_\infty \leq C(\lambda) |u|_{\infty}.$ Indeed, if $M$ is compact then the injectivity radius $\inj(M)$ is positive and we can choose $D_0= B_{\inj(M)/4}(x)$ and $D= B_{\inj(M)/2}(x)$, in which case $r_0 =\inj(M)/4$. Conversely, if $M$ is non-compact then for each $x_0 \in M$ there exist $D_0,D$ with $x_0 \in D_0 \subset D$ and $r_0 = 1$. Either way, $du$ is bounded. By the Weitzenb\"{o}ck formula and integration by parts we have that
\begin{equation}
\e^{{\lambda s}/2}\,\left(\frac{t-s}{t}\right)(du)(\mathscr{Q}_sv) +\e^{\lambda s/2}\,\frac{u(X_s(x))}{t}\int_0^s \langle \mathscr{Q}_r v,dB_r\rangle
\end{equation}
is a local martingale. Since $du$ is bounded, it is actually a true martingale on $[0,t]$. Taking expectations at times $0$ and $t$ yields
\begin{equation}
(du)(v)  = \frac1{t}\mathbb{E}\left[ \e^{{\lambda t}/2} u(X_t(x))\int_0^{t} \langle \mathscr{Q}_s v,dB_s\rangle\right]
\end{equation}
from which the desired formula is obtained by setting $t = 2/\lambda$.
\end{proof}

\begin{cor}
Suppose $\Ric\geq K$ and $\Delta u = -\lambda u$ for some $\lambda >0$. Then
\begin{equation}
|du|_\infty  \leq |u|_\infty \, \e\,\sqrt{\frac{\lambda}2} \left( \frac{1-\e^{- 2K /\lambda}}{2K/\lambda}\right)^{1/2}.
\end{equation}
\end{cor}

\section{Integral Estimates}\label{sec:intestimates}

\subsection{} In this section we obtain global integral versions of the above estimates, for the case $Z=0$. We will use them to derive Corollary \ref{cor:meanzero}. The following lemma is well-known:

\begin{lemma}\label{lem:intlem}
Suppose $u\in C^2(M)$ with $u\in L^p(M)$ for some $p>1$. Suppose $\Ric \geq K$ for some $K \in \mathbb{R}$. Set $q = p/(p-1)$. Then
\begin{equation}
\| dP_t u \|_{L^p(M)}\leq \frac{\e^{K^-t/2}}{\sqrt{t}} C_q^{1/q} \, \|u\|_{L^p(M)}
\end{equation}
for all $t>0$, where $C_q$ is the constant from the Burkholder-Davis-Gundy inequality.
\end{lemma}

\begin{proof}
By the Bismut formula with $h(r) = (t-r)/t$, H\"{o}lder's inequality and the Burkholder-Davis-Gundy inequality, we have
\begin{equation}
|dP_t u|(x) \leq \frac{1}{t}\E\left[ u^p(X_{t}(x))\right]^{1/p} \,\E\left[ \left( \int_0^{t} \langle \mathscr{Q}_r,dB_r\rangle \right)^q\right]^{1/q}\leq  \frac{\e^{K^-t/2}}{\sqrt{t}}C_q^{1/q}\E\left[ u^p(X_{t}(x))\right]^{1/p}
\end{equation}
for all $t >0$ and $x \in M$. Consequently
\begin{equation}
\| dP_t u \|_{L^p(D)}\leq \frac{\e^{K^-t/2}}{\sqrt{t}} C_q^{1/q}  \|u\|_{L^p(M)}
\end{equation}
for any regular domain $D$. The result follows by taking an exhausting sequence of such domains.
\end{proof}

A classical result of Strichartz \cite[Corollary~2.5]{Strichartz} states that if $u,\Delta u \in L^2$ then $du \in L^2$. Using a Gagliardo-Nirenberg inequality, G\"{u}neysu and Pigola have recently proved the following extension \cite[Theorem~4]{GPlp}: for a geodesically complete manifold, if $p\in (1,\infty)$ and $u,\Delta u \in L^p$ then
\begin{equation}
\| du \|_p^2 \leq C \|u\|_p \|\Delta u\|_p + \max\lbrace p-2,0\rbrace \|u\|_p\|\nabla^2 u\|_p
\end{equation}
where $C$ is a constant depending only on $p$, provided the right-hand side is finite. Supposing an $L^p$-Calderon-Zygmund inequality and the existence of a sequence of Hessian cut-off functions, they furthermore proved that $du \in L^p$ if and only if $\max\lbrace p-2,0\rbrace \nabla^2 u \in L^p$. Consequently, the Ricci curvature lower bound appearing in the following global integral version of Theorem \ref{thm:mainthm} is actually redundant if $p\in (1,2]$, but in general can be used to derive the zero-mean value condition given by Corollary \ref{cor:meanzero} in the introduction:

\begin{thm}\label{thm:mainthm3}
Suppose $u \in C^2(M)$ with $u,\, \Delta u \in L^p(M)$ for some $p>1$. Suppose $\Ric \geq K$ for some $K \in \mathbb{R}$. Set $q = p/(p-1)$ and $\delta >0$. Then
\begin{equation}\label{eq:localinest}
\| du \|_{L^p(M)} \leq C_q^{1/q} \e^{\frac{K^-}{8\delta^2}} \left( 2\delta\, \| u \|_{L^p(M)} + \frac{1}{2\delta}\, \| \Delta u \|_{L^p(M)}\right)
\end{equation}
where $C_q$ is the constant from the Burkholder-Davis-Gundy inequality.
\end{thm}

\begin{proof}
By It\^{o}'s formula we have
\begin{equation}
du(x) = dP^1_tu(x) - \int_0^t dP^2_s(Lu)(x) \,ds
\end{equation}
and therefore by Minkowski's inequality
\begin{equation}
\| du \|_{L^p(M)} \leq \| dP_t u\|_{L^p(M)} + \int_0^t \| dP_s(Lu)\|_{L^p(M)} \,ds.
\end{equation}
By Lemma \ref{lem:intlem} it follows that
\begin{equation}
\| du \|_{L^p(M)} \leq C_q^{1/q} \e^{K^-t/2} \left( \frac1{\sqrt{t}}\, \| u \|_{L^p(M)} + \sqrt{t}\, \| \Delta u \|_{L^p(M)}\right)
\end{equation}
from which the result follows by setting $t = 1/4\delta^2$.
\end{proof}

\section{Extension to differential forms}\label{sec:forms}

\subsection{ }
Our results can also be extended to differential forms, for which we assume $Z=0$ for simplicity. So $X(x)$ now denotes a Brownian motion on $M$ starting at $x$.

Denote by $\Omega^p(M)=\Gamma(\Lambda^p T^\ast M)$ the space of differential forms of degree $p$, by $d$ the exterior derivative, by $\delta = d^\ast$ the codifferential and consider the Hodge Laplacian $\Delta:=-(d +\delta)^2$, also known as the Laplace-de Rham operator, acting on $\Omega^p(M)$. The Hodge Laplacian $\Delta$ is related to the connection Laplacian $\square = \trace \nabla^2$ by the Weizenb\"{o}ck formula $\Delta = \square - R_p$ where $R_p \in \Gamma(\End (\Lambda^p T^\ast M))$ is a symmetric field of endomorphisms (for which a precise formula is given by, for example, \cite[Prop.~A.7]{DriverThal_2002}). We then define a process $Q$ acting on, say, $\Lambda^q T^\ast_xM$, to be the $\End(\Lambda^q T^\ast_xM)$-valued solution to the ordinary differential equation
\begin{equation}
\frac{d}{dt} Q_t = -\frac12 Q_t //^{-1}_t R_q //_t, \quad Q_0 = \id_{\Lambda^q T^\ast_x M},
\end{equation}
along the paths of $X(x)$. In terms of the process $Q$ and for $\alpha \in \Omega^p(M)$ we then have the two semigroups
\begin{equation}
P_t^1 \alpha(x) = \mathbb{E}\left[Q_{t\wedge \tau} //_{t\wedge \tau}^{-1}\alpha(X_{t\wedge \tau}(x))\right], \quad P_t^2 \alpha(x) = \mathbb{E}\left[1_{\lbrace t < \tau \rbrace} Q_t //_t^{-1}\alpha(X_t(x))\right]
\end{equation}
as before, which solve the diffusion equation $(\partial_t - \frac12\Delta)\alpha_t=0$ on $D$ with $\alpha_0 = \alpha$ and boundary conditions $P^1_t\alpha \vert_{\partial D} = \alpha\vert_{\partial D}$ and $P^2_t\alpha\vert_{\partial D} = 0$, respectively.

Denote by $\mathscr{Q}_t$ the transpose of $Q_t$ and by $\mathscr{R}_p$ the transpose of $R_p$. Further, for $e \in T_xM$ denote by $C_e$ the creation operator (exterior product) and by $A_e$ the annihilation operator (interior product); $A_e$ is the adjoint of $C_e$. Then, for an orthonormal basis $\lbrace e_i \rbrace_{i=1}^n$, we have
\begin{equation}
\mathscr{R}_p = -\sum_{i,j=1}^n R(e_j,e_i)C_{e_j}A_{e_i}
\end{equation}
where $R(e_j,e_i)$ is the curvature tensor acting on $p$-forms (see \cite[Lemma~A.9]{DriverThal_2002}). Then for $\alpha \in \Omega^p(M)$ there are, according to \cite{DriverThal_2002}, the Bismut formulae
\begin{align}
(dP^1_t\alpha)_x(v) &= -\mathbb{E}\left[ \Bigg\langle Q_{{t\wedge \tau}} //_{t\wedge \tau}^{-1} \alpha (X_{t\wedge \tau}(x)),\int_0^{t\wedge \tau} \mathscr{Q}_s^{-1} (A_{dB_s}\mathscr{Q}_s \dot{h}(s) v)\Bigg\rangle \right],\label{eq:derforone}\\
(dP^2_t\alpha)_x(v) &= -\mathbb{E}\left[ \Bigg\langle 1_{\lbrace t < \tau\rbrace}  Q_{t} //_{t}^{-1} \alpha (X_{t}(x)),\int_0^{t} \mathscr{Q}_s^{-1} (A_{dB_s}\mathscr{Q}_s \dot{h}(s) v)\Bigg\rangle \right]\label{eq:derfortwo}
\end{align}
for $v \in \Lambda^{p+1}T_xM$ and
\begin{align}
(\delta P^1_t\alpha)_x(v) &= \mathbb{E}\left[ \Bigg\langle Q_{{t\wedge \tau}} //_{t\wedge \tau}^{-1} \alpha (X_{t\wedge \tau}(x)), \int_0^{t\wedge \tau} \mathscr{Q}_s^{-1} ( C_{dB_s}\mathscr{Q}_s \dot{h}(s) v )\Bigg\rangle \right],\label{eq:divforone}\\
(\delta P^2_t\alpha)_x(v) &= \mathbb{E}\left[ \Bigg\langle 1_{\lbrace t < \tau\rbrace} Q_{t} //_{t}^{-1} \alpha (X_{t}(x)), \int_0^{t} \mathscr{Q}_s^{-1} ( C_{dB_s}\mathscr{Q}_s \dot{h}(s) v )\Bigg\rangle \right]\label{eq:divfortwo}
\end{align}
for $v \in \Lambda^{p-1} T_xM$, where the process $h$ is as in the previous section. Note that $\mathscr{R}_1 = \Ric$ and $\mathscr{R}_0 = 0$, so for $p=0$ formulae \eqref{eq:derforone} and \eqref{eq:derfortwo} reduce to formulae \eqref{eq:bisone} and \eqref{eq:bistwo} (with $Z=0$).

Now set
\begin{align}
\ovl{K}_p &:= \sup \big\lbrace \langle\mathscr{R}_p(v),v\rangle: v \in \Lambda^pTD,\ |v|=1 \big\rbrace \\
\unl{K}_p &:= \inf \big\lbrace \langle\mathscr{R}_p(v),v\rangle: v \in \Lambda^pTD,\ |v|=1 \big\rbrace \\
\unl{K}_{p\pm 1} &:= \inf \big\lbrace \langle\mathscr{R}_{p\pm 1}(v),v\rangle: v \in \Lambda^{p\pm 1}TD,\ |v|=1 \big\rbrace
\end{align}
with the metric is defined on the exterior algebra in the usual way.

\begin{thm}\label{thm:estforforms}
Suppose $D_0 \subset D$ are regular domains with $\alpha\in \Omega^p(M)$. Set $r_0 = \rho(D_0,\partial D)$ and $\delta >0$. Then
\begin{align}
\sup_{D_0}|d\alpha| &\leq C_{p,+}\left( \frac{2\delta}{r_0} \sup_D |\alpha|  + \frac{r_0}{2\delta} \sup_D |\Delta \alpha| \right) \label{eq:derestone}\\
\sup_{D_0}|\delta \alpha| &\leq C_{p,-}\left( \frac{2\delta}{r_0} \sup_D |\alpha|  + \frac{r_0}{2\delta} \sup_D |\Delta \alpha| \right)\label{eq:deresttwo}
\end{align}
where
\begin{equation}
C_{p,\pm}:= \exp\left(\frac{\pi r_0}{16\delta^2} \left( 2 \sup_D |Z|+\sqrt{(n-1)K^-_0}\right) + \frac{\pi^2 (n+3)}{32\delta^2}+\frac{r_0^2 \left(\unl{K}_p+(\ovl{K}_p + \unl{K}_{p\pm 1})^-\right)}{8\delta^2}\right).
\end{equation}
\end{thm}

\begin{proof}
We will prove inequality \eqref{eq:derestone} using formulae \eqref{eq:derforone} and \eqref{eq:divforone}. The proof of inequality \eqref{eq:deresttwo}, using formulae \eqref{eq:derfortwo} and \eqref{eq:divfortwo}, is nearly identical. By It\^{o}'s formula, we have
\begin{equation}
Q_{t\wedge \tau} //_{t\wedge \tau}^{-1}\alpha(X_{t\wedge \tau}(x)) - \alpha(x) = \int_0^{t\wedge \tau} Q_s //_s^{-1}\nabla_{//_s dB_s}\alpha(X_s(x)) + \frac12\int_0^{t\wedge \tau} Q_s //_s^{-1}\Delta\alpha(X_s(x))\,ds.
\end{equation}
Taking expectations and differentiating we obtain
\begin{equation}\label{eq:ineqinthm}
|d\alpha|(x) \leq |d P^1_t \alpha|(x)+\frac12\int_0^t|dP^2_s \Delta \alpha|(x)\,ds.
\end{equation}
By formulae \eqref{eq:derforone} and \eqref{eq:divforone} and Lemmas \ref{lem:est} and \ref{lem:lemball}, as in the proof of Proposition \ref{prop:prelimest}, the result follows (like the proof of Theorem \ref{thm:mainthm} for functions).
\end{proof}

Note that for $p=0$ inequality \eqref{eq:derestone} reduces to \eqref{eq:mainthmest} (with $Z=0$). Generalizations of the integral estimates in Section \ref{sec:intestimates} can also easily be obtained.

\subsection{}
Replacing $\delta$ with $\delta \sqrt{\lambda}$ in Theorem \ref{thm:estforforms} we obtain the following estimates on the exterior derivative and codifferential of eigenforms of the Hodge Laplacian:

\begin{cor}\label{cor:eigestforforms}
Suppose $D_0 \subset D$ are regular domains with $\alpha\in \Omega^p(M)$ satisfying $\Delta \alpha = -\lambda \alpha$ for some $\lambda >0$. Then
\begin{align}
\sup_{D_0} |d\alpha| &\leq  \sqrt{\lambda} C_{p,+}^{1/\lambda}\left( \frac{2\delta}{r_0}+\frac{r_0}{2\delta}\right) \,\sup_D| \alpha|\label{eq:ederestone}\\
\sup_{D_0} |\delta \alpha| &\leq \sqrt{\lambda} C_{p,-}^{1/\lambda}\left( \frac{2\delta}{r_0}+\frac{r_0}{2\delta}\right)\,\sup_D| \alpha|\label{eq:ederesttwo}
\end{align}
where the constants $C_{p,\pm}$, $r_0$ and $\delta$ are given as in Theorem \ref{thm:estforforms}.
\end{cor}

Note that instead of the Hodge Laplacian $\Delta$, we could just as well have worked with the semigroups generated by the connection Laplacian $\square$.

\subsection{}
Our approach to these estimates is easily adapted to different settings. For example, our results can be extended further to the abstract setting of vector bundles considered in \cite{DriverThal_2002}.

\proof[Acknowledgements]This work has been supported by the Fonds
National de la Recherche Luxembourg (FNR) under the OPEN scheme
(project GEOMREV O14/7628746), as well as by PUL AGSDE of the University of Luxembourg.
The first named author acknowledges support by NSFC (Grant No.~11501508)
and Zhejiang Provincial Natural Science Foundation of China (Grant
No. LQ16A010009).

\providecommand{\href}[2]{#2}

\end{document}